\documentclass[reqno]{amsart}
\usepackage{amsmath,amsthm,amssymb,amscd}
\newtheorem{theorem}{Theorem}[section]

\newtheorem{corollary}[theorem]{Corollary}
\theoremstyle{definition}
\newtheorem{definition}[theorem]{Definition}

\theoremstyle{remark}

\newtheorem{conj}[theorem]{Conjecture}
\numberwithin{equation}{section}
\allowdisplaybreaks
\begin{document}
\title[Arithmetic properties of certain $t$-regular partitions]{Arithmetic properties of certain $t$-regular partitions}
	
\author{Rupam Barman}
\address{Department of Mathematics, Indian Institute of Technology Guwahati, Assam, India, PIN- 781039}
\email{rupam@iitg.ac.in}
	
\author{Ajit Singh}
\address{Department of Mathematics, Indian Institute of Technology Guwahati, Assam, India, PIN- 781039}
\email{ajit18@iitg.ac.in}
	
\author{Gurinder Singh}
\address{Department of Mathematics, Indian Institute of Technology Guwahati, Assam, India, PIN- 781039}
\email{gurinder.singh@iitg.ac.in}
	
\date{September 3, 2022}




\dedicatory{}

\begin{abstract} For a positive integer $t\geq 2$, let $b_{t}(n)$ denote the number of $t$-regular partitions of a nonnegative integer $n$. 
Motivated by some recent conjectures of Keith and Zanello, we establish infinite families of congruences modulo $2$ for $b_9(n)$ and $b_{19}(n)$. 
We prove some specific cases of two conjectures of Keith and Zanello on self-similarities of $b_9(n)$ and $b_{19}(n)$ modulo $2$. 
We also relate $b_{t}(n)$ to the ordinary partition function, and prove that $b_{t}(n)$ satisfies the Ramanujan's famous congruences for some infinite families of $t$. 
For $t\in \{6,10,14,15,18,20,22,26,27,28\}$, Keith and Zanello conjectured that there are no integers $A>0$ and $B\geq 0$ for which $b_t(An+ B)\equiv 0\pmod 2$ for all $n\geq 0$. 
We prove that, for any $t\geq 2$ and prime $\ell$, there are infinitely many arithmetic progressions $An+B$ for which $\sum_{n=0}^{\infty}b_t(An+B)q^n\not\equiv0 \pmod{\ell}$. 
Next, we obtain quantitative estimates for the distributions of $b_{6}(n), b_{10}(n)$ and $b_{14}(n)$ modulo 2. 
We further study the odd densities of certain infinite families of eta-quotients related to the 7-regular and $13$-regular partition functions.
\end{abstract}

\maketitle
\section{Introduction and statement of results} 
 A partition of a positive integer $n$ is a finite sequence of non-increasing positive integers $(\lambda_1, \lambda_2, \ldots, \lambda_k)$ such that $\lambda_1+\lambda_2+\cdots +\lambda_k=n$. 
 Let $p(n)$ denote the number of partitions of $n$. If $n=0$, then $p(n)$ is defined to be 1. Let $t\geq 2$ be a fixed positive integer. 
 A $t$-regular partition of a positive integer $n$ is a partition of $n$ such that none of its parts is divisible by $t$. 
 For example, $(6, 4, 3, 2)$ is a 5-regular partition of 15 as none of its parts is divisible by 5. Let $b_{t}(n)$ be the number of $t$-regular partitions of $n$. The generating function for $b_{t}(n)$ is given by 
\begin{align}\label{gen_fun}
\sum_{n=0}^{\infty}b_{t}(n)q^n=\frac{f_{t}}{f_1},
\end{align}
where $f_k:=(q^k; q^k)_{\infty}=\prod_{j=1}^{\infty}(1-q^{jk})$ and $k$ is a positive integer. 
 \par 
 In literature, many authors have studied divisibility and distribution properties of $b_{t}(n)$ for certain values of $t$, see for example \cite{baruahdas, calkin, cui_Gu_2013, cui_Gu_2015, gordon1997, hirschhorn_sellers, Keith2014,Keith2021,xia_yao,zhao_jin_yao}. 
 In a recent paper \cite{Keith2021}, Keith and Zanello studied $t$-regular partitions for certain values of $t\leq 28$, and made several conjectures on $b_t(n)$. 
 In \cite{SB1, SB2}, the first and the second author have proved two conjectures of Keith and Zanello on $b_3(n)$ and $b_{25}(n)$ respectively, and certain specific cases related to their conjectures. 
 Keith and Zanello proved various congruences for $b_9(n)$ and made the following conjecture regarding the self-similarity of $b_9(n)$.
\begin{conj}\cite[Conjecture 9]{Keith2021}\label{conj1}
For all primes $p\equiv \pm 1\pmod 9$, let $\alpha\equiv-3^{-1}\pmod{2p}$, $0<\alpha<2p$, and $\beta= \lfloor \frac{2p}{3}\rfloor$. Then
\begin{align}\label{new-eq-2}
\sum_{n=0}^{\infty}b_{9}(2(pn+\alpha))q^n\equiv q^{\beta}\sum_{n=0}^{\infty}b_{9}(2n+1)q^{pn}\pmod 2.
\end{align}
\end{conj}
In \cite[Theorem 8]{Keith2021}, Keith and Zanello proved some specific cases of Conjecture \ref{conj1} corresponding to $p=17, 19, 37$. In the following theorem, we prove some other specific cases of Conjecture \ref{conj1} corresponding to $p=53, 71, 73, 89$. 
\begin{theorem}\label{thm1}
	If $p\in\left\lbrace  53, 71, 73, 89\right\rbrace$, then
\begin{align*}
\sum_{n=0}^{\infty}b_{9}\left(2pn+\alpha \right)q^{n}\equiv q^{\beta}\sum_{n=0}^{\infty}b_{9}\left(2n+1\right)q^{pn} \pmod2,
\end{align*}
where $\alpha\equiv -3^{-1}\pmod{2p}$, $0<\alpha<2p$, and $\beta=\lfloor \frac{2p}{3}\rfloor$.
\end{theorem}
Keith and Zanello \cite{Keith2021} also studied $2$-divisibility of $b_{19}(n)$ and proved new congruences for the prime $p=5$ and made the following conjecture regarding the self-similarity of $b_{19}(n)$. 
\begin{conj}\cite[Conjecture 11]{Keith2021}\label{conj2}
For a prime $p>3$, let $\gamma\equiv-3\cdot 8^{-1}\pmod{p}$, $0<\gamma<p$, and $\delta= \lfloor \frac{3p}{8}\rfloor$. Then, for a positive proportion of primes $p$, it holds that:
\begin{align}\label{new-eq-3}
\sum_{n=0}^{\infty}b_{19}(2(pn+\gamma))q^n\equiv q^{\delta}\sum_{n=0}^{\infty}b_{19}(2n)q^{pn}\pmod 2.
\end{align}
\end{conj}
In \cite[Theorem 10]{Keith2021}, Keith and Zanello proved that $p=5$ satisfies \eqref{new-eq-3}. It is easy to check that $p=7$ does not satisfy \eqref{new-eq-3}. In the following theorem, we prove that  $p=11$ satisfies \eqref{new-eq-3}.
\begin{theorem} \label{thm2}
We have
\begin{align*}
\sum_{n=0}^{\infty}b_{19}\left(22n+2 \right)q^{n}\equiv q^{4}\sum_{n=0}^{\infty}b_{19}\left(2n\right)q^{11n} \pmod2
\end{align*}
and therefore, for all $k\not\equiv 4\pmod{11}$
	$$b_{19}(242n+22k+2)\equiv0\pmod{2},$$
	and by iteration,
	\begin{align*}b_{19}\left(2\cdot11^{2d}n+2\cdot 11^{2d-1}k+2\cdot11^{2d-2}+90\left(\frac{11^{2d}-1}{120}\right)\right)\equiv 0\pmod{2},
	\end{align*}
	for all $d,k\geq1$ with $k\not\equiv4\pmod{11}$.
\end{theorem}
In addition to the study of congruences modulo 2, a number of congruences for $b_t (n)$ modulo primes other than 2 have been proven, see for example \cite{baruahdas, cui_Gu_2015, gordon1997, Keith2014}. 
In the following theorem, we relate the $t$-regular partition function $b_{t}(n)$ to the ordinary partition function $p(n)$. 
\begin{theorem}\label{thm5}
	Let $m,a\geq1$ and $b$ be integers. If $p(an+b)\equiv0\pmod{m}$ for all nonnegative integers $n$, then for all positive integers $t$, we have
	\begin{equation*}
	b_{at}(an+b)\equiv0\pmod{m}.
	\end{equation*} 
\end{theorem}
As an immediate consequence of Theorem \ref{thm5}, we find that $b_{t}(n)$ satisfies the Ramanujan's famous congruences for some infinite families of $t$. 
\begin{corollary}\label{cor2}
	For all $k,t\geq1$ and every nonnegative integer $n$, we have
	\begin{align*}
	b_{5^{k}t}\left( 5^{k}n+\delta_{5,k}\right) & \equiv0 \pmod{5^{k}},\\
	b_{7^{k}t}\left(7^{k}n+\delta_{7,k}\right)&\equiv0\pmod{7^{\lfloor k/2\rfloor+1}} ,\\
	b_{{11}^{k}t}\left(11^{k}n+\delta_{11,k}\right)&\equiv0\pmod{11^{k}}, 
	\end{align*}
	where $\delta_{p,k}:=1/24\pmod{p^k}$ for $p=5,7,11$. 
\end{corollary}
In \cite{Keith2014}, Keith studied divisibility of 9-regular partitions by 3. In \cite[Theorem 1]{Keith2014}, he proved that $b_9(4n+3)\equiv 0\pmod 3$ using modular forms techniques. 
In the following theorem, we find the generating function for $b_9(4n+3)$, and as an immediate consequence we obtain that $b_9(4n+3)\equiv 0\pmod 3$.
\begin{theorem}\label{thm6}
	We have
	\begin{align}\label{b9modidentity}
	\sum_{n=0}^{\infty}b_9(4n+3)q^{n}=3\frac{f^{2}_2f^{2}_6f_9}{f^{5}_1}.
	\end{align}
\end{theorem} 
The study of the parity of the coefficients of eta-quotients is one of the most challenging and interesting questions, which has significant applications in the theory of partitions. Given an integral power series $F(q):=\sum_{n=0}^{\infty}a(n)q^n$ and $r\in \{0, 1\}$, we define
\begin{align*}
\delta_r(F, 2; X):=\frac{\#\{n\leq X: a(n)\equiv r \pmod{2}\}}{X}.
\end{align*} 
The power series $F$ is said to have odd density $\delta$ if the limit
\begin{align*}
\lim _{X\rightarrow \infty}\delta_1(F, 2; X)
\end{align*}
exists and is equal to $\delta$. If $F$ has odd density equal to zero, then we say that $F$ is {\it lacunary} modulo 2. In literature, no eta-quotient is known till today whose coefficients have positive odd density.
Using the Serre's seminal work on nondivisibility of coefficients of integral weight modular forms, lacunarity of certain families of eta-quotients can be established.
For example, Cotron et al. \cite{Cotron2020} proved lacunarity of some families of eta-quotients extending the work of Gordon-Ono \cite{gordon1997}. We phrase their theorem as follows:
\begin{theorem}\cite[Theorem 1.1]{Cotron2020}\label{Cotron}
	Let $F(q)=\frac{\prod_{i=1}^{u}f^{r_i}_{\alpha_{i}}}{\prod_{i=1}^{t}f^{s_i}_{\gamma_{i}}}$, and assume that
	\begin{equation*}
	\sum_{i=1}^{u}\frac{r_i}{\alpha_{i}}\geq\sum_{i=1}^{t}s_i\gamma_{i}.
	\end{equation*}
	Then the coefficients of $F$ are lacunary modulo $2$.
\end{theorem}
Keith and Zanello \cite{Keith2021} conjectured the following non-congruences, analogous to Ono's and Radu's theorems for the partition function (see \cite{ono1996, radu2}):
\begin{conj}\cite[Conjecture 15]{Keith2021}\label{conj3}
Let $t\in \{6,10,14,15,18,20,22,26,27,28\}$. We have
\begin{enumerate}
\item For no integers $A>0$ and $B\geq 0$, $b_t(An+ B)\equiv 0\pmod 2$ for all $n\geq 0$.
\item The series $\frac{f_t}{f_1}$ has odd density 1/2.
\end{enumerate}
\end{conj}
In the following theorem, for a given prime $\ell$, we prove that there are infinitely many arithmetic progressions $An+B$ for which $\sum_{n=0}^{\infty}b_t(An+B)q^n\not\equiv 0\pmod{\ell}$.
\begin{theorem}\label{thm3}
Let $\ell$ be a prime and let $t,r$ be positive integers with $t\geq 2$. Then we have
\begin{align*}
\sum_{n=0}^{\infty} b_t(rn+s)q^{n}\not\equiv 0 \pmod {\ell}
\end{align*} 
for all $s\in\left\lbrace 0,1,\ldots,r-1\right\rbrace $ such that
$
s\equiv (t-1)\frac{d^{2}-1}{24} \pmod{r}
$
for some integer $d$ with $\gcd(d,6tr)=1$.
\end{theorem}	
Theorem \ref{thm3} with $\ell=2$ proves Conjecture \ref{conj3} (1) for infinitely many arithmetic progressions.
In Conjecture \ref{conj3} (2), Keith and Zanello claimed that the series $\frac{f_t}{f_1}$ has odd density 1/2 for $t\in \{6,10,14,15,18,20,22,26,27,28\}$. 
In the following two theorems, we obtain quantitative estimates for the distributions of $b_t(n)$ for $t=6,10,14$.
\begin{theorem}\label{thm4}
For large $X$ and $t=6, 10$, we have
\begin{align}\label{b6bound1}
\#\left\lbrace n\leq X : b_{t}\left( n\right)\ is\ even \right\rbrace \gg \sqrt{X}.
\end{align}
\end{theorem}
\begin{theorem}\label{thm4a}
For large $X$, we have
\begin{align*}
\#\left\lbrace n\leq X: b_{14}\left( 2n\right)\ is\ even \right\rbrace \gg \sqrt{X/3}.
\end{align*}
\end{theorem}
In \cite{Keith2021}, Keith and Zanello demonstrated that an odd density may be constant over appropriate infinite families of eta-quotients. 
As an example, they proved that for any nonnegative integer $k$, the odd density of $\frac{f^{9k+2}_{3}}{f^{3k+1}_{1}}$ is same as the odd density of $6$-regular partition function. 
In the following theorem, we prove that the odd density of $\frac{f^{7k+1}_{7}}{f^{k+1}_{1}}$ is same as the odd density of $7$-regular partition function.
\begin{theorem}\label{thm7}
Let $\delta^{\left( k\right)}_7$ denote the odd density of 
\begin{align*}
\frac{f^{7k+1}_{7}}{f^{k+1}_{1}}.
\end{align*}
If $\delta^{\left( k\right)}_7$ exists for any $k\geq 0$ then $\delta^{(k)}_7$ exists for all $k\geq 0$, and its value is independent of $k$. 
In particular, if the odd density of $7$-regular partitions $\delta^{(0)}_{7}$ exists then all of the $\delta^{\left( k\right)}_7$ exist and are equal to $\delta^{(0)}_{7}$.
\end{theorem}
Next, we show that the odd density of $\frac{f^{13k+1}_{13}}{f^{k+1}_{1}}$ is same as the odd density of $13$-regular partition function. More precisely, we have the following result.
\begin{theorem}\label{thm8}
Let $\delta^{\left( k\right)}_{13}$ denote the odd density of 
\begin{equation*}
\frac{f^{13k+1}_{13}}{f^{k+1}_{1}}.
\end{equation*}
If $\delta^{\left( k\right)}_{13}$ exists for any $k\geq 0$ then $\delta^{\left( k\right)}_{13}$ exists for all $k\geq 0$, and its value is independent of $k$. 
In particular, if the odd density of $13$-regular partitions $\delta^{(0)}_{13}$ exists then all of the $\delta^{(k)}_{13}$ exist and are equal to $\delta^{(0)}_{13}$.
\end{theorem}
\section{Preliminaries}
We recall some definitions and basic facts on modular forms in order to make our proofs relatively self-contained. For more details, see for example \cite{koblitz1993, ono2004}. We first define the matrix groups 
\begin{align*}
\text{SL}_2(\mathbb{Z}) & :=\left\{\begin{bmatrix}
a  &  b \\
c  &  d      
\end{bmatrix}: a, b, c, d \in \mathbb{Z}, ad-bc=1
\right\},\\
\Gamma_{0}(N) & :=\left\{
\begin{bmatrix}
a  &  b \\
c  &  d      
\end{bmatrix} \in \text{SL}_2(\mathbb{Z}) : c\equiv 0\pmod N \right\},
\end{align*}
\begin{align*}
\Gamma_{1}(N) & :=\left\{
\begin{bmatrix}
a  &  b \\
c  &  d      
\end{bmatrix} \in \Gamma_0(N) : a\equiv d\equiv 1\pmod N \right\},
\end{align*}
and 
\begin{align*}\Gamma(N) & :=\left\{
\begin{bmatrix}
a  &  b \\
c  &  d      
\end{bmatrix} \in \text{SL}_2(\mathbb{Z}) : a\equiv d\equiv 1\pmod N, ~\text{and}~ b\equiv c\equiv 0\pmod N\right\},
\end{align*}
where $N$ is a positive integer. A subgroup $\Gamma$ of $\text{SL}_2(\mathbb{Z})$ is called a congruence subgroup if $\Gamma(N)\subseteq \Gamma$ for some $N$. 
The smallest $N$ such that $\Gamma(N)\subseteq \Gamma$
is called the level of $\Gamma$. For example, $\Gamma_0(N)$ and $\Gamma_1(N)$
are congruence subgroups of level $N$. 
\par Let $\mathbb{H}:=\{z\in \mathbb{C}: \text{Im}(z)>0\}$ be the upper half of the complex plane. The group $$\text{GL}_2^{+}(\mathbb{R})=\left\{\begin{bmatrix}
a  &  b \\
c  &  d      
\end{bmatrix}: a, b, c, d\in \mathbb{R}~\text{and}~ad-bc>0\right\}$$ acts on $\mathbb{H}$ by $\begin{bmatrix}
a  &  b \\
c  &  d      
\end{bmatrix} z=\displaystyle \frac{az+b}{cz+d}$.  
We identify $\infty$ with $\displaystyle\frac{1}{0}$ and define $\begin{bmatrix}
a  &  b \\
c  &  d      
\end{bmatrix} \displaystyle\frac{r}{s}=\displaystyle \frac{ar+bs}{cr+ds}$, where $\displaystyle\frac{r}{s}\in \mathbb{Q}\cup\{\infty\}$.
This gives an action of $\text{GL}_2^{+}(\mathbb{R})$ on the extended upper half-plane $\mathbb{H}^{\ast}=\mathbb{H}\cup\mathbb{Q}\cup\{\infty\}$. 
Suppose that $\Gamma$ is a congruence subgroup of $\text{SL}_2(\mathbb{Z})$. A cusp of $\Gamma$ is an equivalence class in $\mathbb{P}^1=\mathbb{Q}\cup\{\infty\}$ under the action of $\Gamma$.
\par The group $\text{GL}_2^{+}(\mathbb{R})$ also acts on functions $f: \mathbb{H}\rightarrow \mathbb{C}$. In particular, suppose that $\gamma=\begin{bmatrix}
a  &  b \\
c  &  d      
\end{bmatrix}\in \text{GL}_2^{+}(\mathbb{R})$. If $f(z)$ is a meromorphic function on $\mathbb{H}$ and $\ell$ is an integer, then define the slash operator $|_{\ell}$ by 
$$(f|_{\ell}\gamma)(z):=(\text{det}~{\gamma})^{\ell/2}(cz+d)^{-\ell}f(\gamma z).$$
\begin{definition}
	Let $\Gamma$ be a congruence subgroup of level $N$. A holomorphic function $f: \mathbb{H}\rightarrow \mathbb{C}$ is called a weakly modular form with integer weight $\ell$ on $\Gamma$ if the following hold:
	\begin{enumerate}
		\item We have $$f\left(\displaystyle \frac{az+b}{cz+d}\right)=(cz+d)^{\ell}f(z)$$ for all $z\in \mathbb{H}$ and all $\begin{bmatrix}
		a  &  b \\
		c  &  d      
		\end{bmatrix} \in \Gamma$.
		\item If $\gamma\in \text{SL}_2(\mathbb{Z})$, then there exists an integer $n_{\gamma}$ such that  $(f|_{\ell}\gamma)(z)$ has a Fourier expansion of the form $$(f|_{\ell}\gamma)(z)=\displaystyle\sum_{n\geq n_{\gamma}}a_{\gamma}(n)q_N^n,$$
		with $a_{\gamma}(n_{\gamma})\neq 0$. Here $q_N:=e^{2\pi iz/N}$.
	\end{enumerate}
\end{definition}
If $n_{\gamma}\geq 0$ for all $\gamma\in \text{SL}_2(\mathbb{Z})$, the we call $f$ a modular form of weight $\ell$ on $\Gamma$. 
For a positive integer $\ell$, the complex vector space of all weakly modular forms (resp. modular forms) of weight $\ell$ with respect to a congruence subgroup $\Gamma$ is denoted by $M^{!}_{\ell}(\Gamma)$ (resp. $M_{\ell}(\Gamma)$).
\begin{definition}\cite[Definition 1.15]{ono2004}
	If $\chi$ is a Dirichlet character modulo $N$, then we say that $f\in M^{!}_{\ell}(\Gamma_1(N))$ has Nebentypus character $\chi$ if
	$$f\left( \frac{az+b}{cz+d}\right)=\chi(d)(cz+d)^{\ell}f(z)$$ for all $z\in \mathbb{H}$ and all $\begin{bmatrix}
	a  &  b \\
	c  &  d      
	\end{bmatrix} \in \Gamma_0(N)$.
\end{definition}	
	 The space of such weakly modular forms (resp. modular forms) is denoted by $M^{!}_{\ell}(\Gamma_0(N), \chi)$ (resp. $M_{\ell}(\Gamma_0(N), \chi)$). The relationship between $M^{!}_{\ell}(\Gamma_1(N))$ and these subspaces is given by the following decomposition:
 $$M_{\ell}^{!}\left(\Gamma_{1}(N)\right)=\bigoplus_{\chi} M_{\ell}^{!}\left(\Gamma_{0}(N), \chi\right),$$
 where the direct sum runs over all Dirichlet characters.
In this paper, the relevant modular forms are those that arise from eta-quotients. Recall that the Dedekind eta-function $\eta(z)$ is defined by
\begin{align*}
\eta(z):=q^{1/24}(q;q)_{\infty}=q^{1/24}\prod_{n=1}^{\infty}(1-q^n),
\end{align*}
where $q:=e^{2\pi iz}$ and $z\in \mathbb{H}$. A function $f(z)$ is called an eta-quotient if it is of the form
\begin{align*}
f(z)=\prod_{\delta\mid N}\eta(\delta z)^{r_\delta},
\end{align*}
where $N$ is a positive integer and $r_{\delta}$ is an integer. We now recall two theorems from \cite[p. 18]{ono2004} which will be used to prove our results.
\begin{theorem}\cite[Theorem 1.64]{ono2004}\label{thm_ono1} If $f(z)=\prod_{\delta\mid N}\eta(\delta z)^{r_\delta}$ 
	is an eta-quotient such that $\ell=\frac{1}{2}\sum_{\delta\mid N}r_{\delta}\in \mathbb{Z}$, 
	$$\sum_{\delta\mid N} \delta r_{\delta}\equiv 0 \pmod{24}$$ and
	$$\sum_{\delta\mid N} \frac{N}{\delta}r_{\delta}\equiv 0 \pmod{24},$$
	then $f(z)$ satisfies $$f\left( \frac{az+b}{cz+d}\right)=\chi(d)(cz+d)^{\ell}f(z)$$
	for every  $\begin{bmatrix}
	a  &  b \\
	c  &  d      
	\end{bmatrix} \in \Gamma_0(N)$. Here the character $\chi$ is defined by $\chi(d):=\left(\frac{(-1)^{\ell} s}{d}\right)$, where $s:= \prod_{\delta\mid N}\delta^{r_{\delta}}$. 
\end{theorem}
Suppose that $f$ is an eta-quotient satisfying the conditions of Theorem \ref{thm_ono1} and that the associated weight $\ell$ is a positive integer. If $f(z)$ is holomorphic at all of the cusps of $\Gamma_0(N)$, then $f(z)\in M_{\ell}(\Gamma_0(N), \chi)$. 
The following theorem gives the necessary criterion for determining orders of an eta-quotient at cusps.
\begin{theorem}\cite[Theorem 1.65]{ono2004}\label{thm_ono2}
	Let $c, d$ and $N$ be positive integers with $d\mid N$ and $\gcd(c, d)=1$. If $f$ is an eta-quotient satisfying the conditions of Theorem \ref{thm_ono1} for $N$, then the 
	order of vanishing of $f(z)$ at the cusp $\frac{c}{d}$ 
	is $$\frac{N}{24}\sum_{\delta\mid N}\frac{\gcd(d,\delta)^2r_{\delta}}{\gcd(d,\frac{N}{d})d\delta}.$$
\end{theorem}
We now recall a result of Sturm \cite{Sturm1984} which gives a criterion to test whether two modular forms are congruent modulo a given prime.
\begin{theorem}\label{Sturm}
	Let $p$ be a prime number, and $f(z)=\sum_{n=n_0}^\infty a(n)q^n$ and $g(z)=\sum_{n=n_1}^\infty b(n)q^n$ be modular forms of weight $k$ for $\Gamma_0(N)$ of characters $\chi$ and $\psi$, respectively, where $n_0, n_1\geq 0$. If either $\chi=\psi$ and 
	\begin{align*}
	a(n)\equiv b(n)\pmod p~~ \text{for all}~~ n\leq \frac{kN}{12}\prod_{d~~ \text{prime};~~ d|N}\left(1+\frac{1}{d}\right),
	\end{align*}
	or $\chi\neq\psi$ and 
	\begin{align*}
	a(n)\equiv b(n)\pmod p~~ \text{for all}~~ n\leq \frac{kN^2}{12}\prod_{d~~ \text{prime};~~ d|N}\left(1-\frac{1}{d^2}\right),
	\end{align*}
	then $f(z)\equiv g(z)\pmod p$~~ $(i.e.,~~a(n)\equiv b(n)\pmod p~~\text{for all}~~n)$.
\end{theorem}
\par We next recall the definition of Hecke operators.
Let $m$ be a positive integer and $f(z) = \sum_{n=0}^{\infty} a(n)q^n \in M_{\ell}(\Gamma_0(N),\chi)$. Then the action of Hecke operator $T_m$ on $f(z)$ is defined by 
\begin{align*}
f(z)|T_m := \sum_{n=0}^{\infty} \left(\sum_{d\mid \gcd(n,m)}\chi(d)d^{\ell-1}a\left(\frac{nm}{d^2}\right)\right)q^n.
\end{align*}
In particular, if $m=p$ is prime, we have 
\begin{align*}
f(z)|T_p := \sum_{n=0}^{\infty} \left(a(pn)+\chi(p)p^{\ell-1}a\left(\frac{n}{p}\right)\right)q^n.
\end{align*}
We take by convention that $a(n/p)=0$ whenever $p \nmid n$.
If $f$ is an $\eta$-quotient with the properties listed in Theorem \ref{thm_ono1}, and $p|s$ (here $s$ is as defined in Theorem \ref{thm_ono1}), then $\chi(p)=0$ so that the latter term vanishes. In this case, we have the factorization property that
\begin{align*}
\left(f\cdot\sum_{n=0}^\infty g(n)q^{pn}\right)|T_p=\left(\sum_{n=0}^\infty a(pn)q^{n}\right)\left(\sum_{n=0}^\infty g(n)q^{n}\right).
\end{align*}
\section{Proof of Theorem \ref{thm1}}
\begin{proof}
	We first recall the following identity from \cite[(3.2)]{xia_yao}:
	\begin{align}\label{thmnew}
\sum_{n=0}^{\infty}b_9(2n+1)q^{n}\equiv\frac{f_1f^{2}_9}{f_3}\pmod{2}.
\end{align}
For a given prime $p$, consider
 $$F_{p,1}(z):= \frac{\eta(z)\eta^2(9z)\eta(pz)\eta^2(3pz)\eta(9pz)}{\eta(3z)}$$
and $$F_{p,2}(z):= \frac{\eta(z)\eta^2(3z)\eta(9z)\eta(pz)\eta^2(9pz)}{\eta(3pz)}.$$
By Theorems \ref{thm_ono1} and \ref{thm_ono2}, we find that $F_{p,1}(z)$ and $F_{p,2}(z)$ are modular forms of weight $3$, level $27p$ and the associated character is $\chi_0(\bullet)=(\frac{-3p^2}{\bullet})$. 
For primes $p\equiv-1 \pmod9$, we have $\alpha=\frac{2p-1}{3}$ and $\beta=\alpha$ since $\lfloor \frac{2p}{3}\rfloor=\lfloor \frac{2p-1}{3}\rfloor$. \\
From \eqref{thmnew}, we have
\begin{align*}
F_{p,1}(z)\equiv q^{\alpha+1}\left(\sum_{n=0}^{\infty}b_9(2n+1)q^{n}\right) f_pf^{2}_{3p}f_{9p}\pmod{2}
\end{align*}
and
\begin{align*}
F_{p,2}(z)\equiv q\left(q^{\beta}\sum_{n=0}^{\infty}b_9(2n+1)q^{pn} \right) f_1f^{2}_{3}f_{9}\pmod{2}.
\end{align*}	
 Applying Hecke operator $T_p$ on $F_{p,1}(z)$, we find that
\begin{equation*}
F_{p,1}(z)|T_p\equiv q\left(\sum_{n=0}^{\infty}b_9(2pn+\alpha)q^{n} \right)f_1f^{2}_3f_9\pmod{2}.
\end{equation*}
Since the Hecke operator is an endomorphism on $M_{3}\left(\Gamma_{0}(27p), \chi_0\right)$, we obtain $F_{p,1}(z)|T_{p}\in M_{3}\left(\Gamma_{0}(27p), \chi_0\right)$. 
By Theorem \ref{Sturm}, the Sturm bound for the space $M_{3}\left(\Gamma_{0}(27p), \chi_0\right)$ is $9(p+1)$. For each prime $p\in\left\lbrace53, 71, 89\right\rbrace$, we wish to verify the congruence
\begin{align*}
q\left(\sum_{n=0}^{\infty}b_9(2pn+\alpha)q^{n} \right)f_1f^{2}_3f_9\equiv q^{\alpha+1}\frac{f_pf^{2}_{9p}}{f_{3p}}f_1f^{2}_{3}f_{9}\pmod 2.
\end{align*}
The coefficient of $q^{9(p+1)}$ on the left side involves the value $b_9(18p(p+1)+\alpha)$; thus, $f_9/f_1$ must be expanded at least that far, and the product on the right side must be constructed up to the $q^{9(p+1)}$ terms. 
Finally, expansion with a calculation package such as $Sage$ confirms that all coefficients up to the desired bound are congruent modulo $2$, and we establish the theorem for these three primes.
\par For primes $p\equiv1 \pmod9$, we have $\alpha=\frac{4p-1}{3}$ and $\beta=\frac{\alpha-1}{2}$ since $\lfloor \frac{2p}{3}\rfloor=\lfloor \frac{2p-2}{3}\rfloor$.
Let
$$
G_{p,1}(z):=\frac{\eta(z)\eta^2(9z)\eta^8(pz)}{\eta(3z)}
$$
and
$$
G_{p,2}(z):=\frac{\eta^8(z)\eta^2(9pz)\eta(pz)}{\eta(3pz)}.
$$
Again by Theorems \ref{thm_ono1} and \ref{thm_ono2}, we find that $G_{p,1}(z)$ and $G_{p,2}(z)$ are modular forms of weight $5$, level $27p$ and the associated character is $\chi_1(\bullet)=(\frac{-3p^2}{\bullet})$. \\
From \eqref{thmnew}, we have
\begin{align*}
G_{p,1}(z)\equiv q^{(\alpha+3)/4}\left(\sum_{n=0}^{\infty}b_9(2n+1)q^{n}\right) f^{8}_{p}\pmod{2}
\end{align*}
and
\begin{align*}
G_{p,2}(z)\equiv q\left(q^{\beta}\sum_{n=0}^{\infty}b_9(2n+1)q^{pn} \right) f^{8}_{1}\pmod{2}.
\end{align*}
 Applying Hecke operator $T_p$ on $G_{p,1}(z)$ gives an another element of $M_5\left(\Gamma_0(27p),\chi_1\right)$:
$$
G_{p,2}(z)|T_p\equiv q\left(\sum_{n=0}^{\infty}b_9(2pn+\alpha)q^{n} \right)f^{8}_1\pmod{2}.
$$
By Theorem \ref{Sturm}, the Sturm bound for this space of modular forms is $15(p+1)$. For the prime $p=73$, we wish to verify the congruence
\begin{align*}
q\left(\sum_{n=0}^{\infty}b_9(2pn+\alpha)q^{n} \right)f^{8}_1\equiv q^{(\alpha+1)/2}\frac{f_pf^{2}_{9p}}{f_{3p}}f^{8}_{1}\pmod 2.
\end{align*}
In this case, Sturm bound is $15(p+1)$. Using $Sage$, we verify that all coefficients up to the desired bound are congruent modulo $2$ for the prime $p=73$. This completes the proof of the theorem.
\end{proof}
\section{Proof of Theorem \ref{thm2}}
\begin{proof}
Taking $t=19$ in \eqref{gen_fun}, we obtain
\begin{align}\label{thm2.1}
\sum_{n=0}^{\infty}b_{19}(n)q^n=\frac{f_{19}}{f_1}.
\end{align}
	Let $$H_{11,1}(z):= \frac{\eta(19z)\eta^{126}(121z)}{\eta(z)}$$
	and $$H_{11,2}(z):= \frac{\eta(209z)\eta^{126}(11z)}{\eta(11z)}.$$
	Using Theorems \ref{thm_ono1} and \ref{thm_ono2}, we find that $H_{11,1}(z)$ and $H_{11,2}(z)$ are modular forms of weight $63$, level $2299$ and the associated character is $\chi_2(\bullet)=(\frac{-19\times11^{2}}{\bullet})$. We next calculate that 
	\begin{align*}
H_{11,1}(z)|T_2|T_{11}\equiv q^{29}\sum_{n=0}^{\infty}b_{19}(22n+2)q^{n}f^{63}_{11}\pmod{2}
\end{align*}
and
\begin{align*}
H_{11,2}(z)|T_2\equiv q^{33}\sum_{n=0}^{\infty}b_{19}(2n)q^{11n}f^{63}_{11}\pmod{2}.
\end{align*}
	By Theorem \ref{Sturm}, the Sturm bound for the space $M_{63}\left(\Gamma_0(2299),\chi_2\right)$ is $13860$. We wish to verify the congruence
	\begin{align*}
	q^{29}\left(\sum_{n=0}^{\infty}b_{19}(22n+2)q^{n}\right)f^{63}_{11}\equiv q^{33}\left(\sum_{n=0}^{\infty}b_{19}(2n)q^{11n}\right)f^{63}_{11}\pmod 2.
	\end{align*} 
Expansion with a calculation package such as $Sage$ confirms that all coefficients up to the desired bound are congruent modulo $2$, and the first part of the theorem is established.\\
	Since only powers for which $11|n-4$ can be non-zero on the right side of the statement, therefore, for all $k\not\equiv 4\pmod{11}$ we obtain:
	\begin{align*}
	b_{19}(22(11n+k)+2)=b_{19}(2\cdot 11^2n+22k+2)\equiv 0\pmod2.
	\end{align*}
Now, recursively applying the relation 
	\begin{align*}
	b_{19}(2n)\equiv b_{19}(2\cdot 11^2 n+90)\pmod 2,
	\end{align*}
	we obtain
	\begin{align*}
	&b_{19}(2\cdot 11^2n+22k+2)\\
	&\equiv b_{19}(2\cdot 11^2(11^2n+11k+1)+90)\pmod 2\\
	&=b_{19}(2\cdot 11^4n+2\cdot11^3k +2\cdot 11^2+90)\\
	&\equiv b_{19}(2\cdot 11^6n+2\cdot11^5k+2\cdot 11^4+2\cdot 11^2\cdot 45 +90)\pmod 2\\
	&\equiv \ldots\\
	&\equiv b_{19}\left(2\cdot11^{2d}n+2\cdot 11^{2d-1}k+2\cdot11^{2d-2}+90\left(\frac{11^{2d}-1}{120}\right)\right)\equiv 0\pmod{2},
	\end{align*}
	for all $d,k\geq1$ with $k\not\equiv4\pmod{11}$. We note that the last line is given by a finite geometric summation. This completes the proof of the theorem.
\end{proof} 
\section{Proof of Theorems \ref{thm5} and \ref{thm6}}
\begin{proof}[Proof of Theorem \ref{thm5}]
	We know that the generating function for the partition function $p(n)$ is given by
	$$\sum_{n=0}^{\infty}p(n)q^n=\frac{1}{(q;q)_{\infty}}.$$
	From \eqref{gen_fun}, we find that
	\begin{align}
	&\sum_{n=0}^{\infty}b_{t}(n)q^n\nonumber \\
	&=\frac{(q^{t};q^{t})_{\infty}}{(q;q)_{\infty}}\nonumber\\
	&=\sum_{n=0}^{\infty}p(n)q^{n}\left(\sum_{n\in \mathbb{Z}}(-1)^{n}q^{t n(3n-1)/2} \right)\nonumber\\
	&=\sum_{n=0}^{\infty}p(n)q^{n}\left(1+\sum_{u=1}^{\infty}(-1)^{u}q^{t u(3u-1)/2}+\sum_{v=1}^{\infty}(-1)^{v}q^{t v(3v+1)/2}\right)\nonumber\\
	&=\sum_{n=0}^{\infty}\left( p(n)+\sum_{u=1}^{\infty}(-1)^{u}p\left(n-\frac{t u(3u-1)}{2} \right) +\sum_{v=1}^{\infty}(-1)^{v}p\left(n-\frac{t v(3v+1)}{2} \right)\right) q^{n}.
	\end{align}
	Thus, for all nonnegative integers $n$, we have
	\begin{align}\label{thm6.1}
	b_{t}(n)=p(n)+\sum_{u=1}^{\infty}(-1)^{u}p\left(n-\frac{t u(3u-1)}{2} \right) +\sum_{v=1}^{\infty}(-1)^{v}p\left(n-\frac{t v(3v+1)}{2} \right).
	\end{align}
	In \eqref{thm6.1}, replacing $n$ with $an+b$ and $t$ with $at$, we obtain
	\begin{align}\label{thm6.2}
	&b_{at}(an+b)\nonumber\\
	&=p(an+b)+\sum_{u=1}^{\infty}(-1)^{u}p\left(a\left(n-\frac{tu(3u-1)}{2}\right)+b \right)\nonumber\\ &\hspace{2cm}+\sum_{v=1}^{\infty}(-1)^{v}p\left(a\left(n-\frac{tv(3v+1)}{2}\right) +b \right).
	\end{align}
	We observe that the terms remaining in the sums in \eqref{thm6.1} satisfy that $n-\frac{t u(3u-1)}{2}$ and $n-\frac{t v(3v+1)}{2}$ are nonnegative. Hence, the same is true in \eqref{thm6.2}. Now, if $p(ad +b )\equiv 0 \pmod m$ for every
	nonnegative integer $d$, then \eqref{thm6.2} yields that $b_{at}(an + b)\equiv 0 \pmod m$. This completes the proof of the theorem.
\end{proof}
\begin{proof}[Proof of Corollary \ref{cor2}]
	The Ramanujan's most celebrated congruences for ordinary partition function are given as follows. For $k\geq1$ and any nonnegative integer $n$, we have
	\begin{align*}
	p\left( 5^{k}n+\delta_{5,k}\right) & \equiv0 \pmod{5^{k}},\\
	p\left(7^{k}n+\delta_{7,k}\right)&\equiv0\pmod{7^{\lfloor k/2\rfloor+1}} ,\\
	p\left(11^{k}n+\delta_{11,k}\right)&\equiv0\pmod{11^{k}}.
	\end{align*}
	Combining the above congruences and Theorem \ref{thm5} we readily obtain that
	$b_{at}(n)$ satisfies the Ramanujan congruences when $a = 5^k, 7^k, 11^k$.	
\end{proof}
\begin{proof}[Proof of Theorem \ref{thm6}]
	We first recall the following 2-dissections:
	\begin{align}\label{dissectionf9}
	\frac{f_9}{f_1}&= \frac{f^{3}_{12}f_{18}}{f^{2}_{2}f_{6}f_{36}}+q\frac{f^{2}_4f_{6}f_{36}}{f^{3}_2f_{12}},\\\label{f3byf1}
	\frac{f_3}{f^{3}_1}&= \frac{f^{6}_{4}f^{3}_{6}}{f^{9}_{2}f^{2}_{12}}+3q\frac{f^{2}_4f_{6}f^{2}_{12}}{f^{7}_2},\\\label{dissectionf1f3}
	\frac{1}{f_1f_3}&= \frac{f^{2}_{8}f^{5}_{12}}{f^{2}_{2}f_{4}f^{4}_{6}f^{2}_{24}}+q\frac{f^{5}_{4}f^{2}_{24}}{f^{4}_{2}f^{2}_{6}f^{2}_8f_{12}}.
	\end{align}
	Identity \eqref{dissectionf9} is Lemma $3.5$ in \cite{xia_yao}, \eqref{f3byf1} is in \cite{naika_gireesh}, and \eqref{dissectionf1f3} is equation (30.12.3) in \cite{hirschhorn}. We next rewrite \eqref{dissectionf9} as 
	\begin{align}\label{thm6.01}
	\frac{f_9}{f_1}= \frac{f^{3}_{12}}{f_{36}}\left(\frac{f_{18}}{f_{2}} \right)\left(\frac{1}{f_{2}f_{6}} \right) +q\frac{f^{2}_4f_{36}}{f_{12}}\left(\frac{f_{6}}{f^{3}_{2}} \right).
	\end{align}
	Magnifying equations \eqref{dissectionf9}-\eqref{dissectionf1f3} by $q\rightarrow q^{2}$, and substituting resultants in \eqref{thm6.01}, and then extracting terms with powers of $q$ congruent to $3$ modulo $4$, we obtain \eqref{b9modidentity}. This completes the proof of the theorem.
\end{proof}	
\section{Proof of Theorem \ref{thm3}}
We will need the following theorem of Garthwaite and Jameson \cite{garthwaite_jameson} for the proof of Theorem \ref{thm3}. Let $B\in \mathbb{Z}$, $k\in \frac{1}{2}\mathbb{Z}$, and $N\in\mathbb{Z}^{+}$. Set
\begin{align*}
\mathcal{S}(B,k,N,\chi):=\left\lbrace\eta^{B}(\tau)F(\tau):F(\tau)\in M^{!}_{k}(\Gamma_{0}(N),\chi) \right\rbrace,
\end{align*}
where $M^{!}_{k}(\Gamma_{0}(N),\chi)$ is the space of weakly modular forms of weight $k$ and level $N$ with character $\chi$.
\begin{theorem}\cite[Theorem 1]{garthwaite_jameson}\label{garthwaite_incong}
Let $\ell$ be prime, let $f(\tau)=q^{B/24}\sum_{n\geq n_0}a(n)q^{n}\in \mathcal{S}(B,k,N,\chi)$ have rational $\ell$-integral coefficients, and let $u\in\mathbb{Z}^{+}$. Let $v_{0}\in \mathbb{Z}$ such that
\begin{align*}
\ell \nmid a(v_{0})\hspace{10pt} and \hspace{10pt}a(n)=0\ \text{for all}\ n<v_0\ \text{with}\ n\equiv v_{0} \pmod{u}.
\end{align*}
Then for all $v\in \left\lbrace0,\ldots,u-1\right\rbrace $ such that
\begin{align*}
v\equiv v_0d^{2}+B\frac{d^{2}-1}{24}\pmod{u}
\end{align*}
for some integer $d$ with $\gcd(d,6Nu)=1$, we have that
\begin{align*}
\sum a(un+v)q^{n}\not\equiv0\pmod{\ell}.
\end{align*}
\end{theorem}	
\begin{proof}[Proof of Theorem \ref{thm3}]
Consider
\begin{equation*}
f(\tau):=q^{\frac{t-1}{24}}\sum_{n=0}^{\infty}b_t(n)q^{n}=\eta^{t-1}(\tau)\frac{\eta(t\tau)}{\eta^{t}(\tau)}.
\end{equation*}
As per the notations in Theorem \ref{garthwaite_incong}, we have $B=t-1$, $k=-\frac{t-1}{2}$, and $N=24t$, i.e., $f(\tau)\in \mathcal{S}(t-1,-\frac{t-1}{2},24t,\chi)$ for some Nebentypus character $\chi$. 
The proof follows applying Theorem \ref{garthwaite_incong} for $v_{0}=0$, since $b_t(0)=1$.
\end{proof}
\section{Proof of Theorems \ref{thm4} and \ref{thm4a}} 
In \cite{ahlgren}, Ahlgren found quantitative estimates for the distribution of parity of the ordinary partition function $p(n)$ in arithmetic progression. We follow a similar approach to prove Theorems \ref{thm4} and \ref{thm4a}.
\begin{proof}[Proof of Theorem \ref{thm4}] 
We recall, Euler's Pentagonal Number Theorem \cite[(1.3.18)]{berndt},
\begin{align}\label{Pentagonal}
f_1=\sum_{n\in \mathbb{Z}}^{}(-1)^{n}q^{\frac{n}{2}(3n-1)}
\end{align}
and equation (2) of \cite{Keith2021}
\begin{align*}
\frac{f^{3}_{3}}{f_1}\equiv \sum_{n\in \mathbb{Z}}^{}q^{n(3n-2)} \pmod{2}.
\end{align*}
Using the above identities in
\begin{align*}
\sum_{n=0}^{\infty}b_6(n)q^{n}\equiv \frac{f^{2}_{3}}{f_{1}}\pmod{2},
\end{align*}
we find that 
\begin{align}\label{mainqnb6}
\sum_{n\in \mathbb{Z}}^{}q^{n(3n-2)}\equiv\sum_{n=0}^{\infty}b_6(n)q^{n}\sum_{n\in \mathbb{Z}}q^{3n(3n-1)/2}\pmod{2}.
\end{align}
Clearly for
\begin{equation*}
\sum_{n=0}^{\infty}a(n)q^{n}:=\sum_{n\in \mathbb{Z}}^{}q^{n(3n-2)},
\end{equation*}
we have
\begin{equation}\label{b6ox}
\#\left\lbrace n\leq X : a(n)\ \text{is odd}\right\rbrace=o(X).
\end{equation}
Set $u_k:=\frac{3}{2}k(3k-1)$, $k\in \mathbb{Z}$. For every nonnegative integer $n$, define the set
\begin{equation*}
\mathcal{M}_{n}:=\left\lbrace n-u_k:0\leq u_k\leq n,\ \text{for some}\ k\in \mathbb{Z} \right\rbrace.
\end{equation*}
Now comparing the coefficients of $q^{n}$ on both sides of \eqref{mainqnb6}, we obtain
\begin{equation}\label{mainproofb6}
a(n)\equiv\sum_{m\in\mathcal{M}_n}b_6(m) \pmod{2}.
\end{equation}
Note that for $k\geq1$, if $u_{-(k-1)}\leq n< u_k$, then $|\mathcal{M}_n|=2k+1$ and if $u_{k}\leq n< u_{-k}$, then $|\mathcal{M}_n|=2k$. 
Thus, $|\mathcal{M}_n|$ is odd if and only if $n$ is in an interval of the form $I_{k}:=\left[u_{-(k-1)},u_{k}\right)$. There exists a positive constant $C$ such that $I_{k}\subset\left[0,X \right] $, $0\leq k\leq C\sqrt{X}$, for large $X$. 
The fact that the length of $I_{k}$ is $\gg k$ implies
\begin{equation*}
\#\left\lbrace n\leq X :n\in I_k\  \text{for some}\  k\right\rbrace \gg \sum_{k=0}^{C\sqrt{X}}k\gg X.
\end{equation*}
Therefore, $\#\left\lbrace n\leq X :|\mathcal{M}_n|\ \text{is odd}\  \right\rbrace \gg X$, and together with \eqref{b6ox} we conclude that 
\begin{equation*}
\#\left\lbrace n\leq X :|\mathcal{M}_n|\ \text{is odd,}\ a(n)\ \text{is even} \right\rbrace \gg X.
\end{equation*}
It is clear from \eqref{mainproofb6} that for every $n\in \left\lbrace n\leq X :n\in I_k\  \text{for some}\  k\right\rbrace$, $b_6(m)$ is even for some $m\in \mathcal{M}_n$. This gives
\begin{equation*}
\#\left\lbrace(n,m): n\leq X,\ m\in \mathcal{M}_n,\ b_6(m)\ \text{is even} \right\rbrace \gg X.
\end{equation*}
We now wish to count $M_{m, X}:=\#\left\lbrace n\leq X: m\in \mathcal{M}_n \right\rbrace$. For fixed $m$, $M_{m, X}$ is not more than $\#\left\lbrace k\in\mathbb{Z}: 0\leq u_k\leq X \right\rbrace$, and this number is clearly $\ll \sqrt{X}$. 
Therefore, $M_{m, X} \ll \sqrt{X}$, and we arrive at \eqref{b6bound1}. This completes the proof for $t=6$. 
\par 
For $t=10$, we employ 
\begin{equation*}
f^3_1=\sum_{n=0}^{\infty}(-1)^{n}(2n+1)q^{\frac{n(n+1)}{2}}
\end{equation*}
and (equation (3) from \cite{Keith2021})
\begin{equation*}
\frac{f^5_5}{f_1}\equiv \sum_{n=1}^{\infty}q^{n^2-1}+\sum_{n=1}^{\infty}q^{2n^2-1}+\sum_{n=1}^{\infty}q^{5n^2-1}+\sum_{n=1}^{\infty}q^{10n^2-1}\pmod{2}
\end{equation*}
in 
\begin{equation*}
\sum_{n=0}^{\infty}b_{10}(n)q^{n}\equiv \frac{f^{2}_{5}}{f_{1}}\pmod{2}
\end{equation*}
to obtain
\begin{align*}
&\sum_{n=1}^{\infty}q^{n^2-1}+\sum_{n=1}^{\infty}q^{2n^2-1}+\sum_{n=1}^{\infty}q^{5n^2-1}+\sum_{n=1}^{\infty}q^{10n^2-1}\\
&\equiv \sum_{n=0}^{\infty}b_{10}(n)q^{n}\sum_{n=0}^{\infty}q^{\frac{5n(n+1)}{2}}\pmod{2}.
\end{align*} 
The rest of the proof goes along similar lines as in the case of $t=6$, so we omit the details for reasons of brevity. This completes the proof of the theorem.
\end{proof}
\begin{proof}[Proof of Theorem \ref{thm4a}] 
In order to prove Theorem \ref{thm4a}, first we recall the following congruence from \cite[(3.4)]{baruahdas}:
\begin{align}\label{b7}
\sum_{n=0}^{\infty}b_7(2n+1)q^{n}\equiv f_1f_{14}\pmod{2}.
\end{align}
We know that the generating function for $14$-regular partitions can be written as
\begin{align}\label{b14}
\sum_{n=0}^{\infty}b_{14}(n)q^n\equiv \frac{f_1f_{14}}{f^2_{1}}\pmod{2}.
\end{align}
Invoking \eqref{b7} in \eqref{b14} yields
\begin{align*}
\sum_{n=0}^{\infty}b_{14}(n)q^{n}\equiv\frac{1}{f^2_1}\sum_{n=0}^{\infty}b_7(2n+1)q^{n}\pmod{2}.
\end{align*}
Taking those terms with even powers of $q$ and then replacing $q^2$ with $q$, we obtain
\begin{align}\label{thm4a.1}
\sum_{n=0}^{\infty}b_{14}(2n)q^{n}\equiv\frac{1}{f_1}\sum_{n=0}^{\infty}b_7(4n+1)q^{n}\pmod{2}.
\end{align}
We then use \eqref{Pentagonal} in \eqref{thm4a.1} to obtain 
\begin{align}\label{b14_main}
\sum_{n=0}^{\infty}b_7(4n+1)q^{n}\equiv \sum_{n=0}^{\infty}b_{14}(2n)q^{n}\sum_{n\in \mathbb{Z}}q^{\frac{n}{2}(3n-1)}\pmod{2}.
\end{align}
We define
\begin{align}\label{b14_an}
\sum_{n=0}^{\infty}a(n)q^n:=\sum_{n=0}^{\infty}b_7(4n+1)q^n
\end{align}
and
\begin{align}\label{b14_2n_define}
\sum_{n=0}^{\infty}c(n)q^n:=\sum_{n=0}^{\infty}b_{14}(2n)q^n.
\end{align}
Note that
\begin{align}\label{b14-ox}
\#\{n\leq X: a(n)\ \text{is odd}\}=o(X),
\end{align}
which follows from \eqref{b7} and the fact that the coefficients of $f_1f_{14}$ have odd density zero, which itself follows from Theorem \ref{Cotron}.
\par Set $u_k:=\frac{1}{2}k(3k-1)$, $k\in \mathbb{Z}$. For every nonnegative integer $n$, we define the set
\begin{align*}
\mathcal{M}_n:=\{n-u_k: 0\leq u_k\leq n,\ \text{for some}\ k\in \mathbb{Z}\}.
\end{align*}
Now comparing the coefficients of $q^{n}$ on both sides of \eqref{b14_main}, we obtain
\begin{align}\label{mainproofb14}
a(n)\equiv \sum_{m\in \mathcal{M}_n}c(m)\pmod{2}.
\end{align}
Note that for $k\geq1$, if $u_{-(k-1)}\leq n< u_k$, then $|\mathcal{M}_n|=2k-1$ and if $u_{k}\leq n< u_{-k}$, then $|\mathcal{M}_n|=2k$. Thus, $|\mathcal{M}_n|$ is odd if and only if $n$ is in an interval of the form $I_{k}:=\left[u_{-(k-1)},u_{k}\right)$. 
There exists a positive constant $C$ such that $I_{k}\subset\left[0,X \right] $, $0\leq k\leq C\sqrt{X}$, for large $X$. The fact that the length of $I_{k}$ is $\gg k$ implies
\begin{align*}
\#\{n\leq X: n\in I_k\ \text{for some}\ k\}\gg \sum_{k=0}^{C\sqrt{X}}k\gg X.
\end{align*}
Therefore, $\#\{n\leq X: |\mathcal{M}_n|\ \text{is odd}\gg X\}$, and together with \eqref{b14-ox} we conclude that 
\begin{align*}
\#\{n\leq X: |\mathcal{M}_n|\ \text{is odd},\ a(n)\ \text{is even}\}\gg X.
\end{align*}
It is clear from \eqref{mainproofb14} that for every $n\in \{ n\leq X :n\in I_k\  \text{for some}\  k\}$, $c(m)$ is even for some $m\in \mathcal{M}_n$. This gives
\begin{align*}
\#\{(n, m): n\leq X, m\in \mathcal{M}_n, \ c(m)\ \text{is even}\}\gg X.
\end{align*}
We now wish to count $N_{m, X}:=\#\{ n\leq X: m\in \mathcal{M}_n \}$. For fixed $m$, $N_{m, X}$ is not more than $\#\left\lbrace k\in\mathbb{Z}: 0\leq u_k\leq X \right\rbrace$, and this number is clearly $\ll \sqrt{3X}$. 
Therefore, $\#\{ n\leq X: m\in \mathcal{M}_n \} \ll \sqrt{3X}$, and we arrive at
\begin{align*}
\#\{m\leq X: c(m)\ \text{is even}\}\gg \sqrt{X/3},
\end{align*}
that is, 
\begin{align*}
\#\{n\leq X: b_{14}(2n)\ \text{is even}\}\gg \sqrt{X/3}.
\end{align*}
This completes the proof of the theorem.
\end{proof}
\section{Proof of  Theorems \ref{thm7} and \ref{thm8}}
\begin{proof}[Proof of Theorem \ref{thm7}]
We first recall the following identity from \cite[(9)]{Keith2021}:
\begin{align}\label{7regular}
\frac{f_7}{f_1}\equiv f^{6}_1+qf^{2}_1f^{4}_7+q^2\frac{f^{8}_7}{f^{2}_1}\pmod{2},
\end{align}
which yields
\begin{align}
\frac{f^{8}_7}{f^{2}_1}&\equiv f^{5}_1f^{7}_7+qf_1f^{11}_7+q^2\frac{f^{15}_7}{f^{3}_1}\pmod{2},\\
\frac{f^{15}_7}{f^{3}_1}&\equiv f^{4}_1f^{14}_7+qf^{18}_7+q^2\frac{f^{22}_7}{f^{4}_1}\pmod{2},\\
\frac{f^{22}_7}{f^{4}_1}&\equiv f^{3}_1f^{21}_7+q\frac{f^{25}_7}{f_1}+q^2\frac{f^{29}_7}{f^{5}_1}\pmod{2},\\
\frac{f^{29}_7}{f^{15}_1}&\equiv f^{2}_1f^{28}_7+q\frac{f^{32}_7}{f^{2}_1}+q^2\frac{f^{36}_7}{f^{6}_1}\pmod{2},\\
\frac{f^{36}_7}{f^{16}_1}&\equiv f_1f^{35}_7+q\frac{f^{39}_7}{f^{3}_1}+q^2\frac{f^{43}_7}{f^{7}_1}\pmod{2},\label{7regularlast}
\end{align}
and for $k\geq 6$
\begin{align}
\frac{f^{7k+1}_7}{f^{k+1}_1}\equiv \frac{f^{7k}_7}{f^{k-6}_1}+q\frac{f^{7k+4}_7}{f^{k-2}_1}+q^2\frac{f^{7k+8}_7}{f^{k+2}_1}\pmod{2}.
\end{align}
Since $\frac{7k}{7}\geq k-6$ and $\frac{7k+4}{7}\geq k-2$ for all $k\geq6$, and the first two terms on the right hand sides of equivalences \eqref{7regular}-\eqref{7regularlast} satisfy the hypothesis of Theorem \ref{Cotron}, 
the first two terms on the right hand side of every equivalence has odd density zero. Therefore, assuming the existence of densities, the left term in each equivalence has the same density as the third term on the right hand side. 
Thus, if the odd density of the eta quotients $\frac{f^{7k+1}_{7}}{f^{k+1}_{1}}$ exists for any $k\geq 0$, then it exists for all $k\geq 0$, and they all are equal to the odd density of $\frac{f_7}{f_1}$ (i.e., when $k=0$), which is $\delta^{(0)}_{7}$.
\end{proof}
\begin{proof}[Proof of Theorem \ref{thm8}]
We first recall the following identity for 13-regular partition from \cite[(3)]{calkin}:
\begin{align}\label{thm8.1}
\frac{f_{13}}{f_1}\equiv f^{3}_4+qf^{5}_2f_{26}+q^6f^{3}_{52}+q^7\frac{f^{7}_{26}}{f_{2}}\pmod{2}.
\end{align}
Using \eqref{thm8.1} and proceeding along similar lines as shown in the proof of Theorem \ref{thm7}, we prove Theorem \ref{thm8}.
\end{proof}
\section{Acknowledgements}
The first and the second author gratefully acknowledge the Department of Science and Technology,
Government of India for the Core Research Grant (Project No. CRG/2021/00314) of SERB.


\begin{thebibliography}{999}

\bibitem{ahlgren}
S. Ahlgren, {\it Distribution of parity of the partition function in arithmetic progressions}, Indag. Mathem. 12 (1999), 173--181.
		
\bibitem{baruahdas}
N. D. Baruah and K. Das, {\it Parity results for 7-regular and 23-regular partitions}, Int. J. Number Theory 11 (2015), 2221--2238.

\bibitem{berndt}
B. C. Berndt, {\it Number theory in the Spirit of Ramanujan}, American Mathematical Society, providence, RI, 2006.

\bibitem{calkin}
N. Calkin, N. Drake, K. James, S. Law, P. Lee, D. Penniston, and J. Radder, {\it Divisibility properties of the $5$-regular and $13$-regular partition functions}, Integers (2008), A60, 10 pp.

\bibitem{Cotron2020}
T. Cotron, A. Michaelsen, E. Stamm and W. Zhu, {\it Lacunary eta-quotients modulo powers of primes}, Ramanujan J. 53 (2020), 269--284.

\bibitem{cui_Gu_2013}
S. -P. Cui and N. S. S. Gu, {\it Arithmetic properties of $\ell$-regular partitions}, Adv. in Appl. Math. 51 (2013), 507--523.

\bibitem{cui_Gu_2015}
S. -P. Cui and N. S. S. Gu, {\it Congruences for $9$-regular partitions modulo $3$}, Ramanujan J. 38 (2015), 503--512.

\bibitem{gordon1997}
B. Gordon and K. Ono, {\it Divisibility of certain partition functions by powers of primes}, Ramanujan J. 1 (1997), 25--34.

\bibitem{garthwaite_jameson}
S. A. Garthwaite and M. Jameson, {\it Incongruences for modular forms and applications to partition functions}, Adv. Math. 376 (2021) 107448.

\bibitem{hirschhorn}
M. D. Hirschhorn, {\it The Power of $q$, a Personal Journey}, Developments in Mathematics, 9 Springer, Cham (2017).

\bibitem{hirschhorn_sellers}
M. D. Hirschhorn and J. A. Sellers, {\it Elementary proofs of parity results for $5$-regular partitions}, Bull. Aust. Math. Soc. 81 (2010), 58--63.


\bibitem{Keith2014}
W. J. Keith, {\it Congruences for 9-regular partitions modulo 3}, Ramanujan J. 35 (2014), 157--164.

\bibitem{Keith2021}
W. J. Keith and F. Zanello, {\it Parity of the coefficients of certain eta-quotients}, Journal of Number Theory 235 (2022), 275--304.

\bibitem{koblitz1993}
N. Koblitz, {\it Introduction to elliptic curves and modular forms}, Springer-Verlag, New York (1991).


\bibitem{ono1996} 
K. Ono, {\it Parity of the partition function in arithmetic progressions}, J. Reine Angew. Math. 472 (1996), 1--15.

\bibitem{ono2004}
K. Ono, {\it The web of modularity: arithmetic of the coefficients of modular forms and $q$-series}, CBMS Regional Conference Series in Mathematics, 102, Amer. Math. Soc., Providence, RI, 2004.

\bibitem{naika_gireesh}
M. S. M. Naika and D. S. Gireesh, {\it Congruences for $3$-regular partitions with designated summands}, Integers 16 (2016) A25, 14 pp.

\bibitem{radu2}
S. Radu, {\it A proof of Subbarao's conjecture}, J. Reine Angew. Math. 672 (2012), 161--175.


\bibitem{SB1}
A. Singh and R. Barman, {\it Divisibility of certain $\ell$-regular partitions by 2}, Ramanujan J. (2022), DOI:10.1007/s11139-022-00580-6.

\bibitem{SB2}
A. Singh and R. Barman, {\it Proofs of some conjectures of Keith and Zanello on $t$-regular partition}, 	arXiv:2201.07046 [math.NT].


\bibitem{Sturm1984} 
J. Sturm, {\it On the congruence of modular forms}, Springer Lect. Notes Math. 1240 (1984), 275--280.


\bibitem{xia_yao}
E. X. W. Xia and O. X. M. Yao, {\it Parity results for 9-regular partitions}, Ramanujan J. 34 (2014), 109--117.


\bibitem{zhao_jin_yao}
T. Y. Zhao, J. Jin, and O. X. M. Yao, {\it Parity results for $11$-, $13$- and $17$-regular partitions}, Coll. Mathematicum 151 (2018), 97--109.

\end{thebibliography}
\end{document}